\documentclass[]{amsart}
\usepackage{amsmath}
\usepackage{amssymb}
\usepackage{amsthm}
\usepackage{amsfonts}
\usepackage{amscd}
\usepackage[abbrev]{amsrefs}
\usepackage{bm} 
\usepackage[dvipdfmx]{graphicx}
\usepackage{enumerate}
\usepackage{url}
\usepackage{epsfig}
\usepackage{color}
\usepackage{float}
\usepackage{setspace}
\usepackage{comment}
\usepackage{appendix}
\usepackage[hang,small,bf]{caption}
\usepackage[subrefformat=parens]{subcaption}
\usepackage[bookmarks=true,%
    colorlinks=true,%
    linkcolor=blue,%
    citecolor=blue,%
    filecolor=blue,%
    menucolor=blue,%
    urlcolor=blue,%
    breaklinks=true]{hyperref}
\usepackage{slashed}    
\usepackage{verbatim}
\usepackage[normalem]{ulem}  
\usepackage{diagbox}

\theoremstyle{definition}
\newtheorem{theorem}{Theorem}[section]

\newtheorem{lemma}[theorem]{Lemma}
\newtheorem{corollary}[theorem]{Corollary}
\newtheorem{proposition}[theorem]{Proposition}

\newtheorem{question}{Question}
\newtheorem*{theorem*}{Theorem}
\newtheorem{thmintro}{Theorem}

\theoremstyle{remark}
\newtheorem{remark}{Remark}[section]

\numberwithin{equation}{section}
\numberwithin{figure}{section}

\makeatletter
\providecommand{\leftsquigarrow}{%
  \mathrel{\mathpalette\reflect@squig\relax}%
}
\newcommand{\reflect@squig}[2]{%
  \reflectbox{$\m@th#1\rightsquigarrow$}%
}
\makeatother

\begin{document}

\title{Combinatorial description of closed $3$-manifolds via ordered
  ideal triangulations}

\author{Stavros Garoufalidis}
\address{
  International Center for Mathematics, Department of Mathematics \\
  Southern University of Science and Technology \\
  Shenzhen, China \newline
  {\tt \url{http://people.mpim-bonn.mpg.de/stavros}}}
\email{stavros@mpim-bonn.mpg.de}
\author{Rinat Kashaev}
\address{Section de Math\'ematiques, Universit\'e de Gen\`eve \\
2-4 rue du Li\`evre, Case Postale 64, 1211 Gen\`eve 4, Switzerland \newline
         {\tt \url{http://www.unige.ch/math/folks/kashaev}}}
\email{Rinat.Kashaev@unige.ch}
\author{Sakie Suzuki}
\address{Department of Mathematical and Computing Science, School of Computing,
Institute of Science Tokyo,
2-12-1, Ookayama, Meguro-ku, Tokyo 152-8552, Japan \newline
{\tt \url{https://sakietotera.com}}}
\email{sakie@comp.isct.ac.jp}


\date{28 May 2026}

\begin{abstract}
It is well known that every compact oriented 3-manifold admits an ideal triangulation,
and that any two such triangulations with at least two ideal tetrahedra are related by
a sequence of Pachner $2$--$3$ moves.
Motivated by constructions in quantum topology, we give a combinatorial description
of closed $3$-manifolds in terms of ordered ideal triangulations and ordered
Pachner $2$--$3$ and $0$--$2$ moves.
\end{abstract}

\maketitle

{\footnotesize
\tableofcontents
}


\section{Introduction}
\label{sec.intro}

\subsection{Ideal triangulations and ordered ideal triangulations}

An ideal triangulation of a compact $3$-manifold $M$ with non-empty boundary
is a decomposition of $\operatorname{int}(M)$ into a collection of ideal tetrahedra
with their faces glued in pairs, where an ideal tetrahedron is obtained by removing
the four vertices of a tetrahedron.
Ideal triangulations are combinatorial objects that were introduced and used by
Thurston to describe the hyperbolic structure of knot complements~\cite{Thurston},
and have since become a fundamental tool in the study of the geometry, topology
and quantum topology of $3$-manifolds. 
Moreover, ideal triangulations have been implemented in software such as
SnapPy~\cite{snappy} and Regina~\cite{regina},
providing effective tools for the classification of $3$-manifolds
and allowing for experimentation and visualization of phenomena
in $3$-dimensional topology.

A fundamental problem in this approach is to determine which local moves
relate ideal triangulations representing the same $3$-manifold.
In particular, it is known that any two ideal triangulations of the same
$3$-manifold with at least two ideal tetrahedra are related by a sequence of
Pachner $2$--$3$ moves and their inverses~\cite{Mat}; see Figure~\ref{fig:Pachner}.

\begin{figure}[H]
    \centering
    \includegraphics{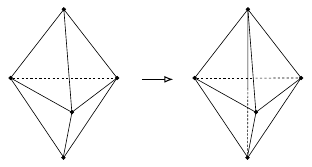}
    \caption{Pachner $2$--$3$ move. }    
    \label{fig:Pachner}
\end{figure}

In this paper, we consider \emph{ordered ideal triangulations}
of compact, oriented $3$-manifolds,
that is, ideal triangulations in which each tetrahedron
is equipped with a total ordering of its vertices
and the face identifications preserve this ordering.

An ordered tetrahedron has two possible types, as shown in Figure~\ref{fig:ordered_tetrahedron}, according to whether the vertex ordering is compatible
with the orientation of the $3$-manifold.

\begin{figure}[H]
    \centering
  \includegraphics{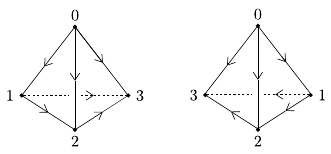}
  \caption{Two types of ordered ideal tetrahedra. The order is indicated
    by numbers on the vertices or by arrows (or both) pointing from smaller
    to larger vertices.}
    \label{fig:ordered_tetrahedron}
\end{figure}

Ordered triangulations arise naturally in several contexts.
In quantum topology, one assigns an operator to each tetrahedron using the ordering
structure.
For example, in the state-sum construction of the Turaev--Viro invariant~\cite{TV},
one assigns a $6j$-symbol to each tetrahedron, where the $6j$-symbol is obtained
from an operator associated with an ordered tetrahedron via a symmetrization procedure
(cf.~\cite{BS2}).
In the construction of quantum invariants of framed $3$-manifolds~\cite{MST1, MST2},
one assigns to each ideal tetrahedron an operator given by the canonical element
of the Heisenberg double of a Hopf algebra.
In particular, in quantum Teichm\"uller theory~\cite{AK}
and in the quantum hyperbolic invariant~\cite{BS0,BS1,BS2},
one assigns to each ideal tetrahedron an operator derived from the Faddeev--Kashaev
quantum dilogarithm,
where the quantum dilogarithm arises from the canonical element of the Heisenberg
double of the quantum Borel subalgebra $\mathrm{U}_q(\mathfrak{sl}_2^+)$.
All of these operators satisfy pentagon-type relations
corresponding to the Pachner $2$--$3$ move, just as the Yang--Baxter equation
corresponds to the Reidemeister III move in quantum link invariants.

\subsection{Ordered ideal triangulations for closed $3$-manifolds}

An ordered ideal triangulation determines not only a 3-manifold, but also a
nowhere-vanishing vector field on it whose flow follows the ordering of the vertices
of each tetrahedron, as shown in Figure~\ref{fig:flow}.

\begin{figure}[H]
    \centering
   \includegraphics{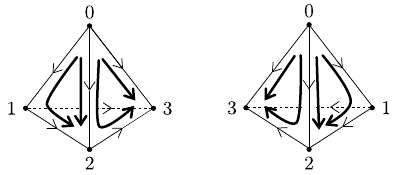}
   \caption{Ordered ideal tetrahedra with flows following the vertex ordering.}
    \label{fig:flow}
\end{figure}

An ideal triangulation on $M$ determines a triangulation of its boundary $\partial M$
by truncating the tetrahedra at their vertices. An ordered triangulation determines
a collection of arcs at these triangles, namely the tangency locus of the
above-mentioned vector field (shown in thick lines in
Figure~\ref{fig:closedness_condition}). The arcs join to a collection of multicurves
on $\partial M$, which we call informally a \emph{sutured structure} on
$\partial M$.

In what follows, we assume that $3$-manifolds are connected.
To describe a closed $3$-manifold $M$, we consider ideal triangulations of
$M \setminus \operatorname{int}(B^3)$.  \footnote{Such ideal triangulations
  are also known as $1$-vertex triangulations of $M$.}
In this setting, we restrict to ordered ideal triangulations with a standard
sutured $S^2$ boundary as shown in Figure~\ref{fig:closedness_condition}.
Under this condition, the nowhere-vanishing vector field extends uniquely from
$M \setminus \operatorname{int}(B^3)$ to $M$.

\begin{figure}[H]
    \centering
   \includegraphics{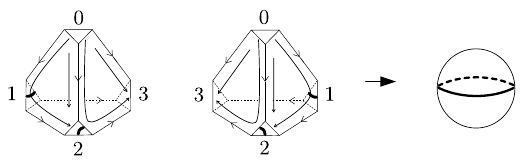}
   \caption{The standard sutured structure on $S^2$: the multicurve is a single
     circle bounding two disks in $S^2$.}
    \label{fig:closedness_condition}
\end{figure}


Such vector fields on closed $3$-manifolds have been studied from a combinatorial
viewpoint via spines. In particular, Ishii~\cite{Ishii2} gave a combinatorial
description of closed $3$-manifolds using \emph{flow spines}, which are spines
arising from non-singular flows on closed $3$-manifolds. 
There is a natural bijection between standard flow spines, where standard means
special in the sense of \cite{Mat}, and ordered ideal triangulations with the
standard sutured $S^2$ boundary. Benedetti and
Petronio~\cite{BP} also gave combinatorial descriptions of \emph{combed} $3$-manifolds
and closed $3$-manifolds using \emph{closed normal o-graphs}, which provide graphical
presentations of ordered ideal triangulations with the standard sutured $S^2$ boundary.

Muramatsu, Taguchi, and the third author~\cite{MST} simplified the combinatorial
descriptions of closed $3$-manifolds by Benedetti--Petronio \cite{BP}. However,
the description still involves the \emph{combinatorial Pontryagin move} (in short, \emph{CP move}),
whose combinatorial structure is rather complicated.

In the present paper we replace the CP move with simpler local moves and obtain a
local combinatorial description of closed $3$-manifolds. In this
introduction we formulate our main result in the language of ordered ideal
triangulations, while further results and proofs will be carried out using the
language of normal o-graphs.

\subsection{Moves for ordered ideal triangulations}

To give a combinatorial description of $3$-manifolds, we consider two types
of local moves on ordered ideal triangulations.

An \emph{ordered $2$--$3$ move} is the Pachner $2$--$3$ move on the underlying
triangulations such that the ordering of the vertices on each of the six boundary
faces is preserved. See Figure~\ref{fig:ordered_23} for the standard ordered
$2$--$3$ move.
\begin{figure}[H]
    \centering
    \includegraphics{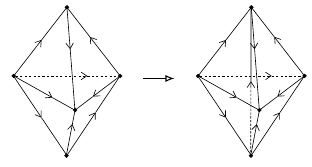}
    \caption{The standard ordered $2$--$3$ move.}
    \label{fig:ordered_23}
\end{figure}

\noindent There are $20$ types of ordered $2$--$3$ moves. There are $16$
combinatorial patterns before the move, and for $4$ of them there are exactly
two possible orientations of the common edge of the $3$ tetrahedra.  

An \emph{ordered $0$--$2$ move} inserts, along a pair of  faces that share an edge,
a ball triangulated by two ordered tetrahedra sharing two faces; see
Figure~\ref{fig:ordered_02}. The ordering structure on the four triangles at the
boundary of the inserted ball must agree with that of the corresponding faces.

\begin{figure}[H]
    \centering
   \includegraphics{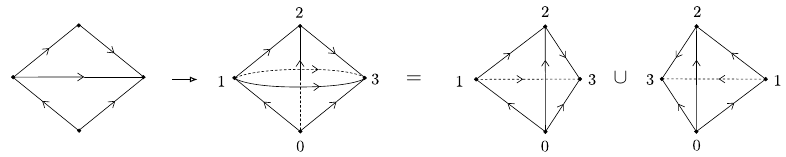}
   \caption{An ordered $0$--$2$ move. The right-hand side
     consists of two ordered tetrahedra glued along two faces $012$ and $023$.}
    \label{fig:ordered_02}
\end{figure}

\noindent For a given pair of faces, the result of the ordered $0$--$2$ move is
unique, and there are $6$ combinatorial patterns, see Figure~\ref{fig:ordered_02_all}.

\begin{figure}[H]
    \centering
   \includegraphics[scale=0.5]{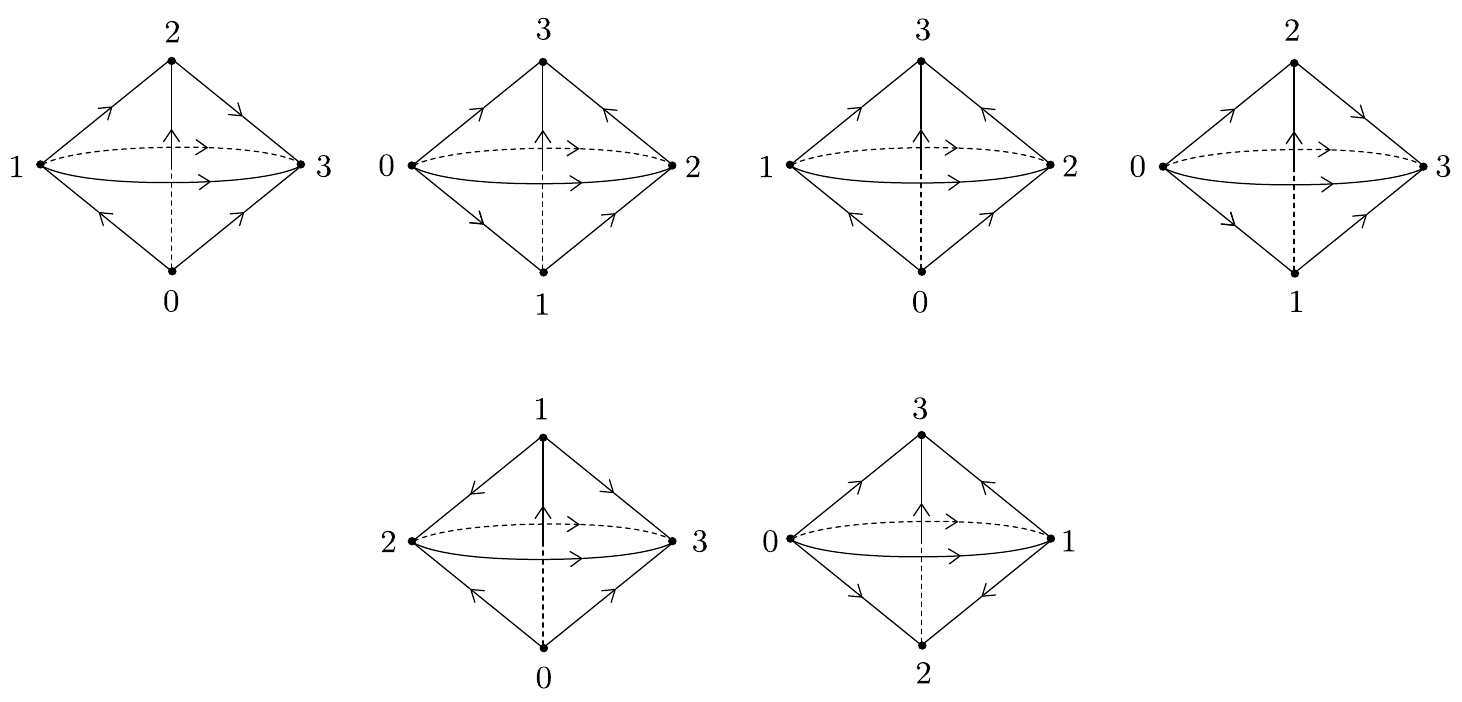}
   \caption{All $6$ combinatorial types of ordered $0$--$2$ moves.}

    \label{fig:ordered_02_all}
\end{figure}

\subsection{Main results}
\label{sub.results}

We arbitrarily choose one ordered $2$--$3$ move and call it the preferred move.  
The following theorem generalizes the result of~\cite{MST}.

\begin{thmintro}[Theorem \ref{MST-thm} and Corollary \ref{bMP}]
\label{main0}
Each ordered $2$--$3$ move can be realized as a sequence
consisting of a single preferred ordered $2$--$3$ move (or its inverse)
together with ordered $0$--$2$ moves and their inverses.
\end{thmintro} 
The main result of this paper is the following.

\begin{thmintro}[Theorem \ref{P}]\label{main}
The equivalence classes of ordered ideal triangulations with the standard
sutured $S^2$ boundary, up to a single preferred ordered
$2$--$3$ move and ordered $0$--$2$ moves, are in one-to-one correspondence with
the orientation-preserving homeomorphism classes of closed $3$-manifolds.
\end{thmintro}

We remark that the two ordered $0$--$2$ moves in the bottom row of
Figure~\ref{fig:ordered_02_all} change the sutured structure on the boundary
(cf. \cite[Figure 3.23]{BP}). The equivalence relation in Theorem~\ref{main} 
is the restriction of the equivalence relation on all ordered ideal triangulations
to those with the standard sutured $S^2$ boundary.

Theorem~\ref{main} refines~\cite[Theorem 3.1]{C}, and gives an affirmative answer
to a restricted version of~\cite[Question 9.1.4]{BP} for closed normal o-graphs,
see Section~\ref{O-graph} for details. 

A natural generalization of Theorem~\ref{main}  to compact $3$-manifolds with torus
boundary is an important problem.

\begin{question}
Consider ordered ideal triangulations with torus boundary, up to ordered $2$--$3$
moves and ordered $0$--$2$ moves. Do their equivalence classes correspond
bijectively to the orientation-preserving homeomorphism classes of compact $3$-manifolds with torus boundary?
\end{question}

This question is a reformulation of~\cite[Question 9.1.3]{BP}
in terms of ordered ideal triangulations, in the torus boundary case.
\subsection{Organization of the paper}

The paper is organized as follows. In Section~\ref{Section;Normal o-graphs}, we
explain a diagrammatic description of ordered ideal triangulations using normal
o-graphs. In Section~\ref{moves}, we describe the moves on ordered ideal
triangulations in terms of normal o-graphs.
In Section~\ref{O-graph}, we state the main result (Theorem~\ref{P}) in terms of normal o-graphs. 
Section~\ref{rel} is devoted to a detailed study of the precise relations among
moves of normal o-graphs.
In Section~\ref{Proof}, we prove Proposition~\ref{P0},
which completes the proof of Theorem~\ref{P}
(equivalently, Theorem~\ref{main} in the introduction).

\subsection*{Acknowledgments}

We would like to thank the organizers of the workshop on ``Low-dimensional Topology
and Number Theory'', held at Oberwolfach in April 2026 for providing a stimulating
environment.
The work of RK is partially supported by the SNSF research program NCCR The
Mathematics of Physics (SwissMAP), the SNSF grants no.~200021-232258, no.~200020-200400,
and no.~10009199.
The work of SS was partially supported by JSPS KAKENHI Grant Number JP24K06736.


\section{Ordered ideal triangulations and normal o-graphs}
\label{Section;Normal o-graphs}

A normal o-graph encodes a branched spine, which is dual to an ordered ideal
triangulation.
In this section we describe ordered ideal triangulations directly in terms of
normal o-graphs, without referring to  branched spines.
For the theory of branched spines, see \cite{BP}.

A \textit{normal o-graph} is a finite connected $4$-valent graph $\Gamma$
with at least one vertex,
equipped with the following additional structures:
\begin{itemize}
\item[(1)] Each vertex of $\Gamma$ is endowed with an embedding of a neighborhood
of the vertex into $\mathbb{R}^2$, viewed up to planar isotopy.
In addition, one pair of opposite edge germs is marked as passing over
the other pair, as in a link diagram.

\item[(2)] Each edge is oriented, and the orientations of opposite edges
agree at each vertex.
\end{itemize}

Such a graph admits a planar diagram in $\mathbb{R}^2$
that realizes the prescribed local planar structures at the vertices.
In general, virtual crossings are introduced as artifacts of the planar representation.
There are two types of classical crossings, positive and negative, as for link
diagrams. Two such planar diagrams represent the same normal o-graph
if they are related by planar isotopy and the Reidemeister-type moves shown
in Figure~\ref{fig:RM}.

\begin{figure}[H]
    \centering
    
    \includegraphics{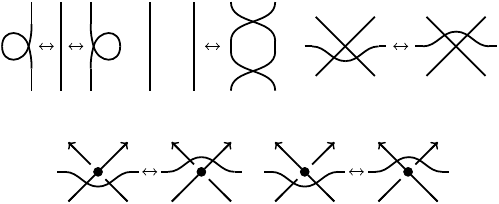}
\caption{Reidemeister-type moves.}
    \label{fig:RM}
\end{figure}

There is a one-to-one correspondence between normal o-graphs
and ordered ideal triangulations, defined as follows.

Let $ \Gamma $ be a normal o-graph.
For each positive (resp. negative) crossing of $ \Gamma $,
we associate an ordered tetrahedron of type $-$
(resp. type $+$) as shown in Figure~\ref{fig:o-graph_vs_triangulaion}, 
so that the four edges of the crossing correspond to the four faces
of the ordered tetrahedron.
We then obtain an ordered ideal triangulation by gluing
the faces of these tetrahedra along the edges of $ \Gamma $.
This gluing is uniquely determined since it is required to preserve
the ordering of vertices.

\begin{figure}[H]
\centering
\includegraphics[scale=0.8]{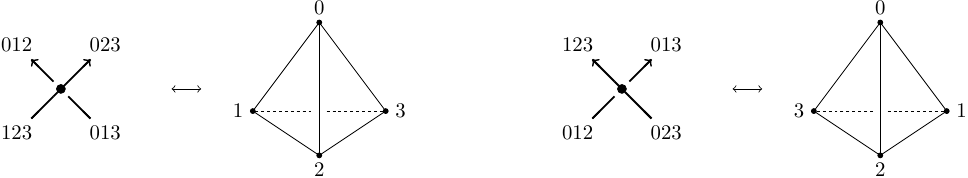}
\caption{Correspondence between crossings of a normal o-graph and ordered tetrahedra.
  A positive crossing  corresponds to an ordered tetrahedron of type $-$ (left),
  and a negative crossing  corresponds to an ordered tetrahedron of type $+$ (right).
  The three numbers at each endpoint of the crossing indicate the face of the
  corresponding ordered tetrahedron. The tetrahedra are drawn in $\mathbb{R}^3$.}
\label{fig:o-graph_vs_triangulaion}
\end{figure}

The above construction is reversible; an ordered ideal triangulation
uniquely determines a normal o-graph.

A \textit{closed normal o-graph} is a normal o-graph representing
an ordered ideal triangulation of a closed $3$-manifold $M$ with a ball removed,
such that the associated vector field divides the spherical boundary
into two hemispheres with ingoing and outgoing vectors.
This closedness condition can also be described purely combinatorially
in terms of normal o-graphs; see~\cite{BP}.


\section{Moves on normal o-graphs}
\label{moves}

Recall that there are two types of moves on ordered ideal triangulations:
the ordered \(2\)--\(3\) moves and the ordered \(0\)--\(2\) moves. We now recall
the corresponding moves on normal o-graphs. If the induced move on ordered ideal
triangulations preserves the sutured structure on the boundary, these moves are
called \emph{Matveev--Piergallini (MP) moves} and \emph{pure sliding moves}.
Otherwise, they are called \emph{bumping MP moves} and \emph{bumping pure
sliding moves}.

\subsection{MP moves and bumping MP moves}
\label{23moves}
We translate the $20$ types of ordered $2$--$3$ moves
on ideal triangulations in terms of normal o-graphs.

The \textit{MP moves} on normal o-graphs are shown in Figures~\ref{The MP moves of type A}--\ref{The MP moves of type D}.

\begin{figure}[H]
    \centering
    \includegraphics[scale=0.8]{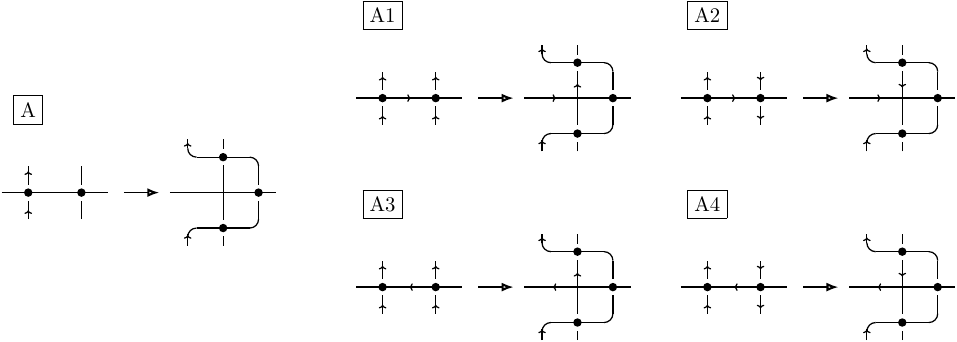}
    \caption{MP moves of type A.}
    \label{The MP moves of type A}
\end{figure}

\begin{figure}[H]
    \centering
    \includegraphics[scale=0.8]{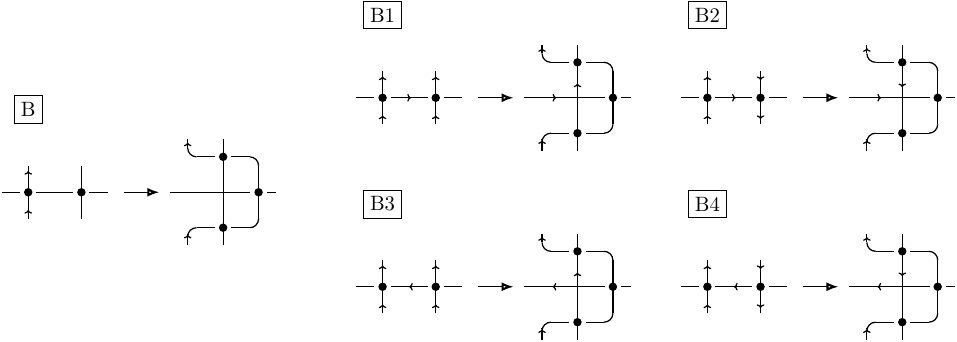}
    \caption{MP moves of type B.}
    \label{The MP moves of type B}
\end{figure}
\begin{figure}[H]
    \centering
    \includegraphics[scale=0.8]{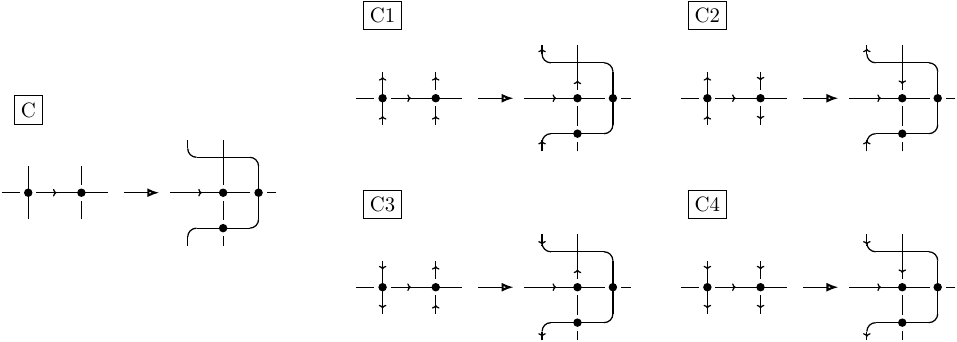}
    \caption{MP moves of type C.}
    \label{The MP moves of type C}
\end{figure}
\begin{figure}[H]
    \centering
    \includegraphics[scale=0.8]{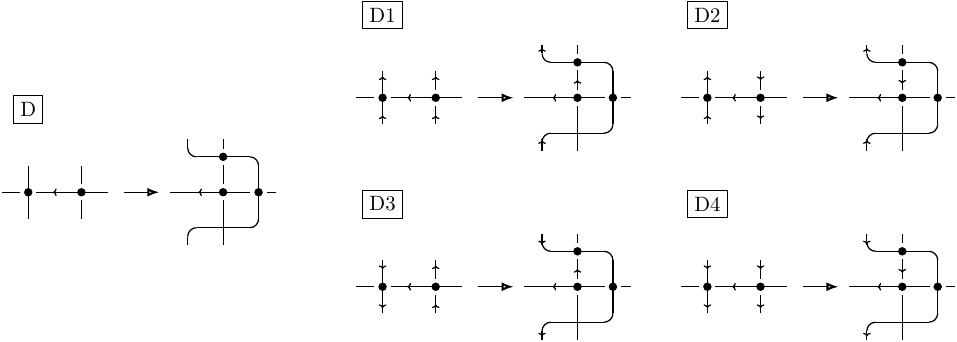}
    \caption{MP moves of type D.}
    \label{The MP moves of type D}
\end{figure}

The \textit{bumping MP moves} on normal o-graphs are shown in Figure~\ref{fig:BMP}.

\begin{figure}[H]
    \centering
    \includegraphics[scale=0.8]{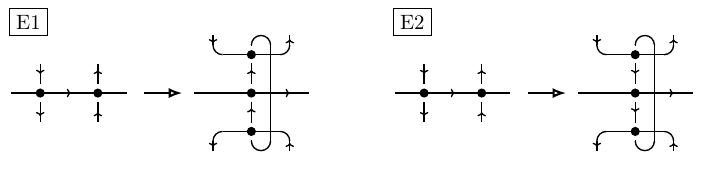}
 
    \includegraphics[scale=0.8]{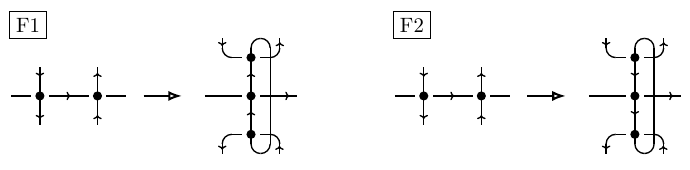}
   \caption{Bumping MP moves.}
    \label{fig:BMP}
\end{figure}

\subsection{Pure sliding moves and bumping pure sliding moves}
\label{02moves}

We translate the $6$ types of ordered $0$--$2$ moves on ideal triangulations
in terms of normal o-graphs: the \textit{pure sliding moves} and the
\textit{bumping pure sliding moves} on normal o-graphs,
shown in Figure~\ref{the pure sliding move} and Figure~\ref{fig:Bps}, respectively.

The local graph on the left-hand side of each move does not contain enough
combinatorial information to determine whether the corresponding two faces in the
ideal triangulation share an edge or how they are attached.
Therefore, we require certain global conditions on normal o-graphs in order to perform these moves.
\footnote{
The pure sliding moves were introduced in~\cite{BP} for branched spines and later
formulated combinatorially on normal o-graphs in~\cite{MST}. Bumping pure sliding
moves for branched spines also appear in~\cite{BP} and are formulated combinatorially
on normal o-graphs in the present paper.
}

\begin{figure}[H]
    \centering
    \includegraphics[scale=0.8]{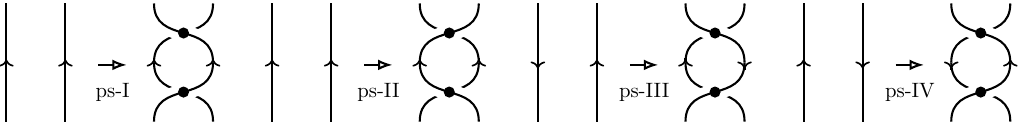}
    \caption{Pure sliding moves ps-I -- ps-IV.}
    \label{the pure sliding move}
\end{figure}

\begin{figure}[H]
    \centering
    \includegraphics[scale=1]{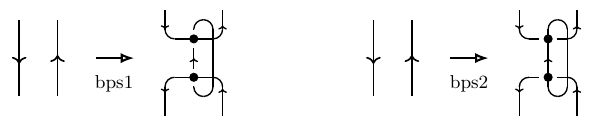}
     \caption{Bumping pure sliding moves.}
    \label{fig:Bps}
\end{figure}

We recall from~\cite{MST} the conditions (PS-I)--(PS-IV) under which ps-I--ps-IV
can be performed on normal o-graphs. 

Given a diagram of a normal o-graph $\Gamma$, we construct immersed oriented closed curves in $\mathbb{R}^2$ as follows.
First, we replace each true vertex of $\Gamma$ with six immersed arcs, as illustrated on the left-hand side of Figure~\ref{fig:circuit}.
Then we connect the endpoints of these arcs by arcs running parallel to the edges, as shown on the right-hand side of Figure~\ref{fig:circuit}.
We call the resulting diagram the \textit{circuit diagram} of $\Gamma$, following the construction in~\cite{BP}, as used also in~\cite{MST}.
For each edge, the three corresponding arcs in the circuit diagram are ordered from left to right when the edge is directed upwards.

\begin{figure}[H]
    \centering
    \includegraphics[scale=1]{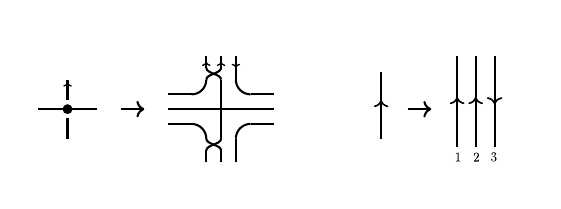}
    \caption{Circuit diagrams around a  vertex (left) and an edge (right).}   
    \label{fig:circuit}
\end{figure}

For each move ps-I--ps-IV, the two edges on which we perform the move must satisfy
the corresponding condition on the circuit diagram, as follows:

\begin{itemize}
\item[\rm (PS-I)]
  the third arc on the left is connected to the second arc on the right;
\item[\rm (PS-II)]
  the third arc on the left is connected to the first arc on the right;
\item[\rm (PS-III)]
  the first arc on the left is connected to the second arc on the right;
\item[\rm (PS-IV)]
  the third arc on the left is connected to the third arc on the right.
\end{itemize}
See Figure~\ref{the pure sliding move condi}.  

\begin{figure}[H]
    \centering
    \includegraphics[scale=0.6]{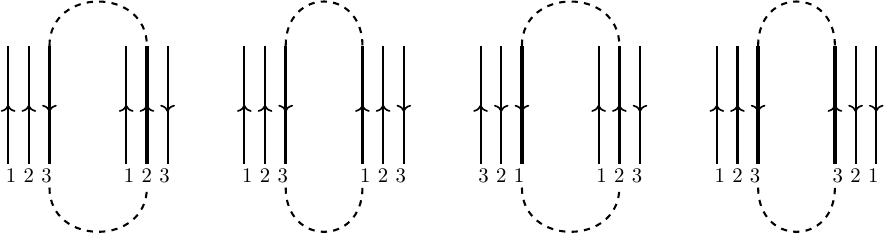}
    \caption{Conditions (PS-I)-- (PS-IV) for ps-I -- ps-IV, respectively.}
    \label{the pure sliding move condi}
\end{figure}

To perform the bumping pure sliding moves bps-1 and bps-2,  we need the
following conditions (BPS1) and (BPS2), respectively.

\begin{itemize}
\item[\rm (BPS1)]
  the second line on the left is connected to the second line on the right,
\item[\rm (BPS2)]
  the first line on the left is connected to the first line on the right.
\end{itemize}

\begin{remark}
The pure sliding moves and the bumping pure sliding moves on branched spines
are shown in Figure~\ref{fig:PS move} and Figure~\ref{fig:Bps move}, respectively.

\begin{figure}[H]
    \centering
    \includegraphics[scale=0.9]{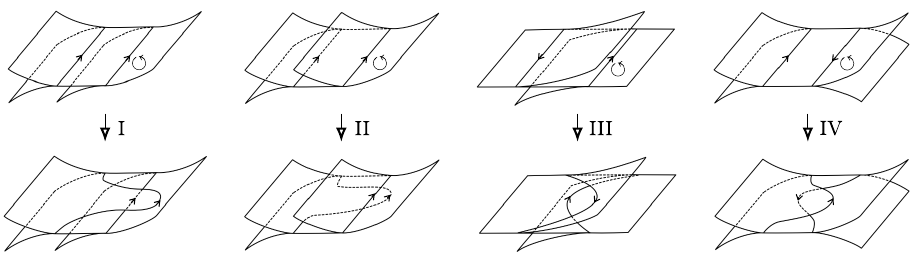}
    \caption{Pure sliding moves I--IV on branched spines.}
    \label{fig:PS move}
\end{figure}

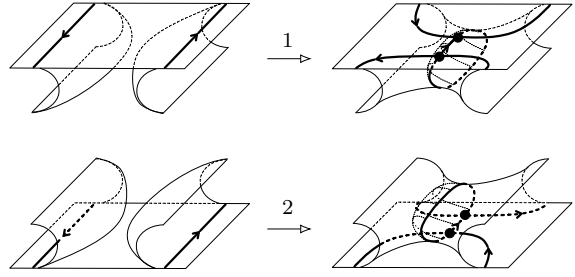
\begin{figure}[H]
    \centering
    \includegraphics[scale=0.28]{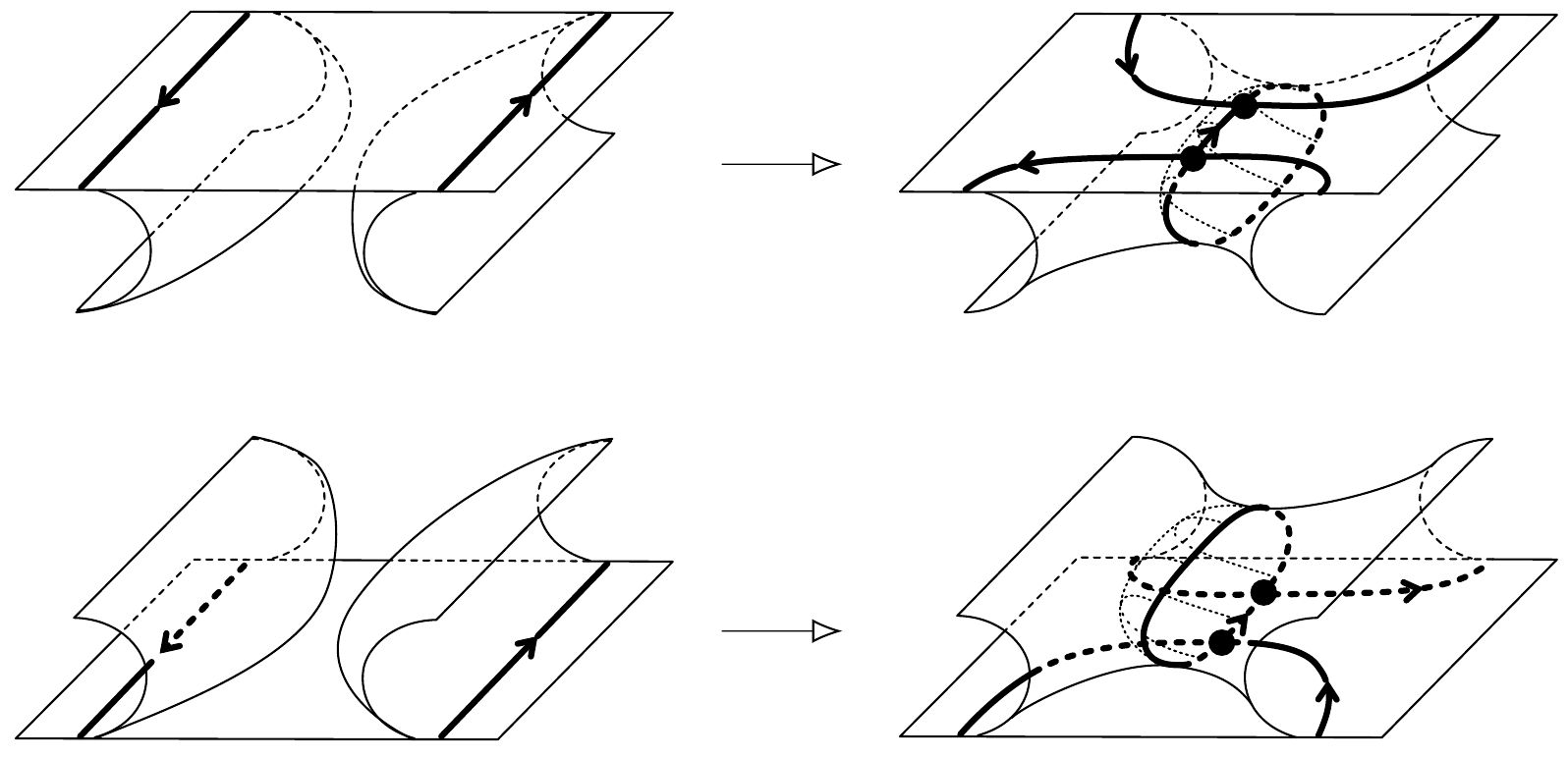}
    
       \begin{picture}(0,0)
         \put(-3,98){\footnotesize $1$}
     \put(-3,35){\footnotesize  $2$ }
      \end{picture}
    
    \caption{Bumping pure sliding moves 1 and 2.}
    \label{fig:Bps move}
\end{figure}

Figure~\ref{fig:correspondence} shows the correspondence between the three germs
of disks at an edge and the three arcs in the circuit diagram. Using this
correspondence, we can check that two edges satisfying each of the conditions
(PS-I)--(PS-IV) (resp.\ BPS1, BPS2) correspond to the left-hand sides of the
pure sliding moves I--IV (resp.\ the bumping pure sliding moves 1 and 2) on branched spines.

\begin{figure}[ht]
    \centering
    \includegraphics[scale=0.7]{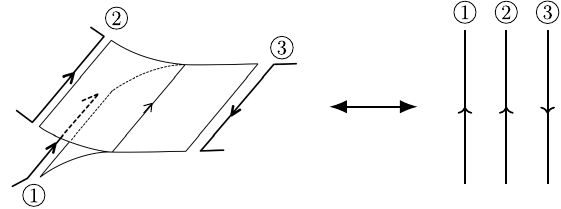}
    \caption{Correspondence between the three germs at an edge and the three
      arcs in the circuit diagram.}
    \label{fig:correspondence}
\end{figure}
\end{remark}

\begin{remark}
The four ordered $0$--$2$ moves shown in the top row of Figure~\ref{fig:ordered_02_all}
correspond to pure sliding moves I--IV (Figures~\ref{the pure sliding move} and~\ref{fig:PS move}),
while the two in the bottom row correspond to bumping pure sliding moves 1 and 2
(Figures~\ref{fig:Bps} and~\ref{fig:Bps move}).
\end{remark}

\subsection{CP move}

The \emph{combinatorial Pontryagin move} (in short, \emph{CP move}) is shown in
Figure~\ref{fig:CP}, which is introduced in \cite{BP} to relate different
combings of the same 3-manifold.
 
\begin{figure}[H]
    \centering
    \includegraphics[scale=0.8]{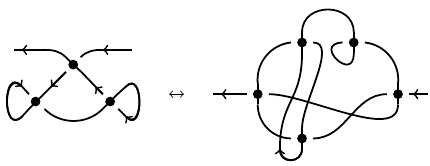}
    \caption{CP move.}
    \label{fig:CP}
\end{figure}


\section{Main results and proofs in terms of normal o-graphs}

\subsection{Main results}
\label{O-graph}

We denote by $\mathcal{M}$ the set of closed $3$-manifolds up to
orientation-preserving diffeomorphism. Benedetti and Petronio~\cite{BP} established
a combinatorial description of $\mathcal{M}$ by defining an equivalence relation
on closed normal o-graphs.

\begin{figure}[ht]
    \centering
    \includegraphics{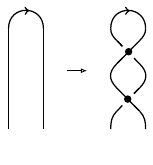}
  \caption{$0$--$2$ move.}
      \label{fig:02}
\end{figure}

\begin{theorem}[\cite{BP}]
\label{BP-thm}
There is a one-to-one correspondence between $\mathcal{M}$ and the set of closed
normal o-graphs, up to  MP moves, the \textit{$0$--$2$ move} shown
in Figure~\ref{fig:02}, and the CP move.
\end{theorem}

Note that the $0$--$2$ move is a special case of a pure sliding move.

Muramatsu, Taguchi, and the third author~\cite{MST}
proved the following.
We choose and fix a preferred MP move among the 16 types of MP moves in Figures~\ref{The MP moves of type A}--\ref{The MP moves of type D}.

\begin{theorem}[\cite{MST}]
\label{MST-thm}
Each MP move can be realized as a sequence
consisting of  a single preferred MP move
(or its inverse), together with pure sliding moves and their inverses.
\end{theorem}

As a consequence, we obtain the following alternative description of $\mathcal{M}$.
 
\begin{corollary}[\cite{MST}]
\label{MST-cor}
There is a one-to-one correspondence between $\mathcal{M}$
and the set of closed normal o-graphs,
up to a preferred MP move, pure sliding moves, and the CP move.
\end{corollary}

It is natural to ask whether a simpler description of $\mathcal{M}$ exists that
does not involve the CP move. 
Benedetti and Petronio asked~\cite[Question 9.1.4]{BP} whether there is a one-to-one correspondence between $\mathcal{M}$ and the equivalence classes of normal o-graphs with $S^2$ boundary generated by the $0$--$2$ move, MP moves, and bumping MP moves. Costantino~\cite[Theorem 3.1]{C} gave an alternative description of $\mathcal{M}$ (in general, for compact oriented $3$-manifolds possibly with boundary) by employing \textit{bubble moves} together with bumping MP moves and bumping pure sliding moves.

We fix a preferred move among the 20 types of MP and bumping MP moves in Figures~\ref{The MP moves of type A}--\ref{fig:BMP}.
We prove the following proposition in Section~\ref{Proof}.

\begin{proposition}
\label{P0}
The CP move can be realized as a sequence of the preferred move, pure
sliding moves, bumping pure sliding moves,  and their inverses.
\end{proposition}

Together with Corollary~\ref{MST-cor}, Proposition~\ref{P0} yields the following description
of closed $3$-manifolds, equivalent to Theorem~\ref{main} in the introduction.

\begin{theorem}
\label{P}
There is a one-to-one correspondence between  $\mathcal{M}$  and the set of closed
normal o-graphs, up to the preferred move, pure sliding moves, and
bumping pure sliding moves.  
\end{theorem}

Theorem~\ref{P} refines Costantino's result~\cite[Theorem 3.1]{C} and also gives an affirmative answer to a restricted version of Benedetti and Petronio's question~\cite[Question 9.1.3]{BP} for closed normal o-graphs.
Indeed, each move used in Theorem~\ref{P} can be obtained as a sequence of the $0$--$2$ move, MP moves, and bumping MP moves, and their inverses (see Remark~\ref{Bfrom}).

\subsection{Relations among moves}
\label{rel}

We describe the precise relations among the moves on normal o-graphs.
Recall from~\cite{MST} that the following relations hold among the MP moves
and their inverses.

\begin{figure}[H]
    \centering
    \includegraphics[scale=1]{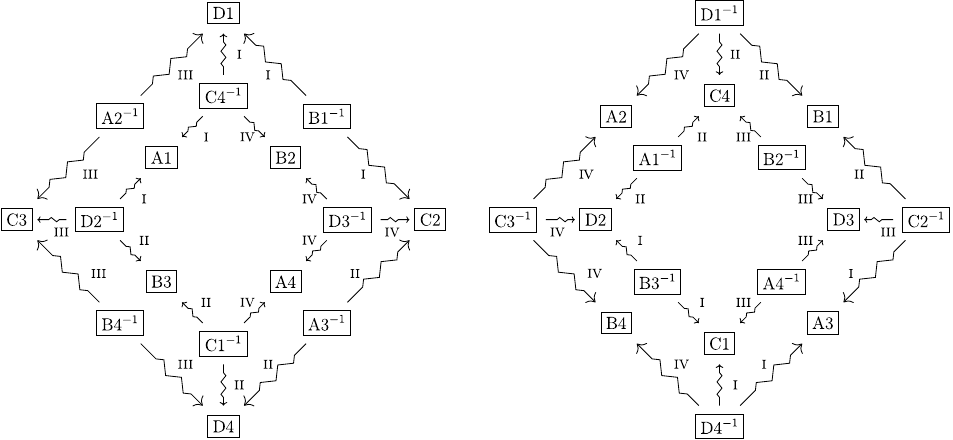}
    \caption{Relations among the MP moves. An arrow from a move $Y$ to a move $X$
      indicates that $X$ is obtained by composing $Y$ with a pure sliding move.
      The labels I--IV attached to the arrows indicate the type of pure sliding move.
      The two MP moves appearing at the same position in the left and right diagrams
      are inverses of each other, and the arrows at corresponding positions point
      in opposite directions.
}
    \label{Relations among MP moves}
\end{figure}

For example, The move A1 is sequence of D$2^{-1}$ and ps-I as shown in
Figure~\ref{D2toA2}.

\begin{figure}[H]
    \centering
    \includegraphics[scale=0.7]{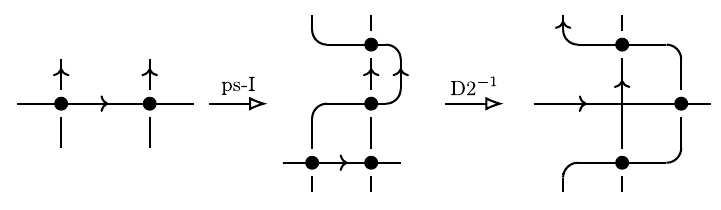}
     \caption{Decomposition of $\mathrm{A1}$ into ps-I followed by $\mathrm{D2}^{-1}$.}
    \label{D2toA2}
\end{figure}

Note that, by tracing the decomposition in the opposite direction, from right to left, 
$\rm A1^{-1}$ decomposes as a sequence of $\rm D2$ followed by the inverse of ps-I.
In Figure~\ref{Relations among MP moves}, this reverse decomposition allows us to
add an arrow in the other diagram at the corresponding position and with the same
direction. Similarly, a reverse arrow can be added for each arrow in
Figure~\ref{Relations among MP moves}. Consequently, each pair of MP moves
(and their inverses) in the same diagram is connected by arrows in both directions.

As observed in~\cite[Remark 2.3]{MST}, the diagram in Figure~\ref{Relations among MP moves}
lacks certain symmetries: in particular, there are no arrows between moves of
A-type and B-type. This reflects the asymmetry of branching structures.
If we also consider bumping moves, the diagram becomes more symmetric.

\begin{lemma}
\label{Lem1}
Relations involving bumping moves are given in Figure~\ref{Relations among bMP moves}.
\end{lemma}

\begin{figure}[H]
    \centering
    \includegraphics[scale=1]{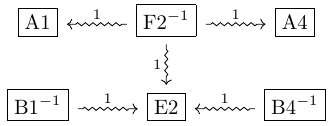} \quad 
   \includegraphics[scale=1]{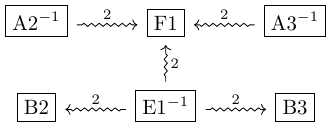} 

   \includegraphics[scale=1]{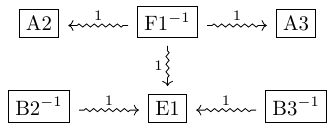} \quad
   \includegraphics[scale=1]{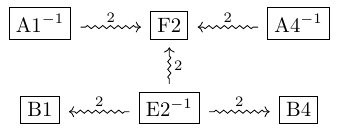}
   \caption{Relation involving bumping moves. An arrow from a move $Y$ to a move
     $X$ indicates that $X$ is decomposed into a bumping pure sliding move followed
     by $Y$. The label 1,2 attached to each arrow indicate the type of bumping
     pure sliding moves.
}
    \label{Relations among bMP moves}
\end{figure}

\begin{proof}
The cases of type $\mathrm{B}^{-1} \rightsquigarrow \mathrm{E}$,
$\mathrm{F}^{-1} \rightsquigarrow \mathrm{E}$, and
$\mathrm{F}^{-1} \rightsquigarrow \mathrm{A}$
are shown in Figures~\ref{EfromB}--\ref{AfromF}, where
the orientations of some edges are omitted and are determined
by the choice of subtype
($\mathrm{A}_1,\dots,\mathrm{A}_4$, $\mathrm{B}_1,\dots,\mathrm{B}_4$,
$\mathrm{E}_1,\mathrm{E}_2$, $\mathrm{F}_1,\mathrm{F}_2$).

The remaining cases, namely
$\mathrm{A}^{-1} \rightsquigarrow \mathrm{F}$,
$\mathrm{E}^{-1} \rightsquigarrow \mathrm{F}$, and
$\mathrm{E}^{-1} \rightsquigarrow \mathrm{B}$,
are obtained similarly by changing the signs of the crossings.

\begin{figure}[H]
    \centering
 \includegraphics[scale=1]{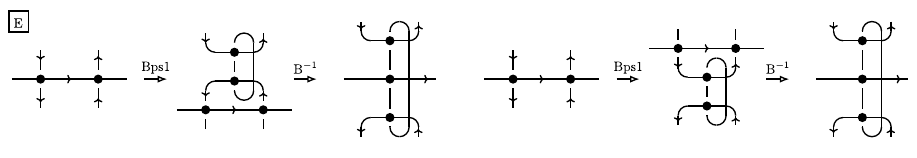}
     \caption{Proof of $\mathrm{B}_2^{-1} \protect\rightsquigarrow \mathrm{E}_1$,
$\mathrm{B}_4^{-1} \protect\rightsquigarrow \mathrm{E}_2$ (left),
$\mathrm{B}_3^{-1} \protect\rightsquigarrow \mathrm{E}_1$,
$\mathrm{B}_1^{-1} \protect\rightsquigarrow \mathrm{E}_2$ (right).}
    \label{EfromB}
\end{figure}

\begin{figure}[H]
    \centering
 \includegraphics[scale=1]{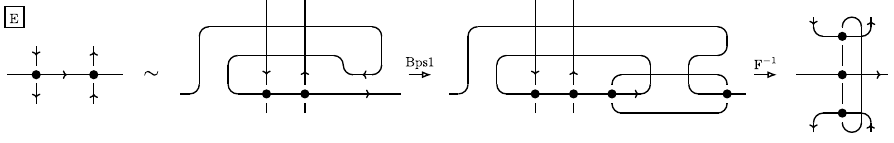}
      \caption{Proof of $\mathrm{F}_1^{-1} \protect\rightsquigarrow \mathrm{E}_1$ and
$\mathrm{F}_2^{-1} \protect\rightsquigarrow \mathrm{E}_2$.
        }
    \label{EfromF}
\end{figure}

\begin{figure}[H]
    \centering
 \includegraphics[scale=1]{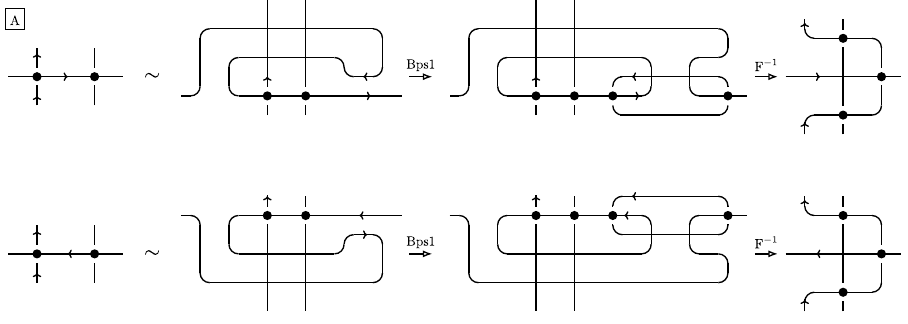}
      \caption{Proof of $\mathrm{F}_2^{-1} \protect\rightsquigarrow \mathrm{A}_1$,
$\mathrm{F}_1^{-1} \protect\rightsquigarrow \mathrm{A}_2$ (top),
$\mathrm{F}_1^{-1} \protect\rightsquigarrow \mathrm{A}_3$,
$\mathrm{F}_2^{-1} \protect\rightsquigarrow \mathrm{A}_4$ (bottom).}
    \label{AfromF}
\end{figure}
\end{proof}

As in Proposition~\ref{P0}, we fix a preferred move among the 20 types of MP and bumping MP moves.

\begin{corollary}
\label{bMP}
Each bumping MP move can be realized as a sequence consisting of a single preferred move
(or its inverse),  together with
pure sliding moves and bumping pure sliding moves, and their inverses.
\end{corollary}

\begin{proof}
We connect the two diagrams at the top (resp.\ bottom) in
Figure~\ref{Relations among bMP moves} to the left (resp.\ right) diagram in
Figure~\ref{Relations among MP moves}. By an argument similar to that used above,
adding reverse arrows, each pair of boxes in
the same diagram is connected by arrows in both directions.

Observe that each type of MP move and bumping MP move appears in either the left
or the right diagram, and its inverse appears in the other. Therefore, each bumping
MP move can be realized as a sequence consisting of the preferred move or its
inverse, together with pure sliding moves, bumping pure sliding moves, and their inverses.
\end{proof}

\begin{remark}
\label{Bfrom}
Benedetti and Petronio~\cite[Lemma~4.5.1]{BP} showed that each pure sliding move can be obtained as a sequence of the $0$--$2$ move, MP moves, and their inverses.
Similarly, one can prove that each bumping pure sliding move can be obtained as a sequence of the $0$--$2$ move, MP moves, bumping MP moves, and their inverses.
\end{remark}

\subsection{Proof  of Proposition~\ref{P0}}
\label{Proof}

By Theorem~\ref{MST-thm} and Corollary~\ref{bMP}, it suffices to express the CP
move as a sequence of MP moves, bumping MP moves, pure sliding moves, and bumping pure sliding moves, together with their inverses.

The sequence realizing the CP move consists of twenty steps. Instead of presenting
all intermediate diagrams, we describe it using the Gauss code (E-datum) of
normal o-graphs.

Let $\Gamma$ be a normal o-graph with ordered $n$ vertices. By removing the
vertices and joining pairs of opposite edges, one obtains a collection of
oriented circuits. We define the number of components of $\Gamma$
to be the number of such circuits.

Assume first that $\Gamma$ has one component. The E-datum
$\mathcal{E}(\Gamma) = [\mathcal{A}; \gamma]$, where $\mathcal{A}$ is a cyclic
permutation of $\{\pm 1, \pm 2, \ldots, \pm n\}$ and
$\gamma \in \{\pm 1\}^n$ is a sequence of vertex types.
The permutation $\mathcal{A}$ records the order in which the vertices
are encountered along this circuit, where the sign indicates whether
the corresponding branch passes over or under a crossing.

For example, the E-datum $[1, 2, -2, -1; 1, 1]$
corresponds to the normal o-graph shown in Figure~\ref{fig:tref}.
\begin{figure}[H]
    \centering
    \includegraphics[scale=1]{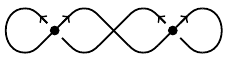}
    
    \begin{picture}(0,0)
    \put(-32,9){$1$}
    \put(27,9){$2$}
    \end{picture}
    
   \caption{Normal o-graph with ordered vertices.}
    \label{fig:tref}
\end{figure}
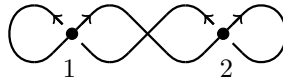

In general, $\Gamma$ may have several components. In this case, the set
$\{\pm 1, \pm 2, \ldots, \pm n\}$ is partitioned according to the circuits,
and the E-data $\mathcal{E}(\Gamma)$ consists of a cyclic ordering on each part,
together with the sequence $\gamma$.

To describe local moves, we also use E-data for graphs with open edges, whose
endpoints are labeled. For example, the E-data of the left- and right-hand sides
of the CP move in Figure~\ref{fig:CP} are given by

\begin{align*}
 \text{LHS of the CP move}&=[[a, -1,2,-2,3,-3,1, b]; [1, 1, 1]], \\
 \text{RHS of the CP move}&=[[a, -1,-2,5,-3,2,3,4,-4,1,-5, b]; [-1, -1, -1, 1, 1]],
\end{align*}

\noindent
where, in both sides, the right and left endpoints are labeled by $a$ and $b$,
respectively.

We now describe a sequence of moves connecting these two E-data.
Each step is specified by the type of move and the vertices to which it is applied.

\medskip
\noindent\textbf{Reduction of LHS of the CP move.}
\begin{align*}
 \text{LHS of the CP move}&=[[a,-1,2,-2,3,-3,1,b]; [1, 1, 1]] 
\\
&\xrightarrow{\mathrm{C3} [2,1]} [[a, 2,-4,-1,-2,3,-3,4,1,b]; [-1, 1, 1, 1]]
\\
&\xrightarrow{\mathrm{C3} [2,3]} [[a, 5,2,-4,-1,3,-5,-2,-3,4,1, b]; [-1, -1, 1, 1, 1]]
\\
&\xrightarrow{\text{(ps-I)}^{-1} [5,2]} [[a, -3,-1,2,-2,3,1, b]; [-1, 1, 1]]
\\
&\xrightarrow{\mathrm{C1} [1,2]} [[a, -3,2,-4,-1,-2,3,1,4,b]; [1, 1, 1, -1]] 
\\
&\xrightarrow{\text{(ps-III)}^{-1} [1,4]} [[a, -2,1,-1,2, b];[1, 1]]
\\
&\xrightarrow{\mathrm{C3} [1,2]} [[a, -1,-2,3,1,2,-3, b]; [-1, 1, 1]]
\\
&\xrightarrow{\mathrm{E2} [1,2]} [[-1,-2,-4],[1],[a, 2,3,4,-3,b]; [1, -1, 1, -1]]
\\
&\xrightarrow{\mathrm{D1} [3,4]} [[-1,-2,-3,-4],[1],[a, 2,5,4,-5,3,b];
[1, -1, -1, -1, 1]]
\\
&\xrightarrow{\mathrm{B2} [4,1]} [[-2,-3,-6],[1,4],[a, 2,5,-4,6,-1,-5,3, b];
[-1, -1, -1, 1, 1, -1]]
\\
&\xrightarrow{\text{(bps-2)}^{-1} [1,4]} [[-1,-2,-4],[4],[a, 1,3,-3,2, b];
[-1, -1, 1, -1]]
\\
&\xrightarrow{\mathrm{E1}[1,3]} [[-3,-1,-5],[4],[a, 5,1,-2,-4,3,2, b];
[-1, -1, 1, -1, 1]]
\\
&\xrightarrow{\text{(ps-IV)}^{-1} [5,1]} [[-2],[3],[a, -1,-3,2,1, b]; [-1, 1, -1]]
\\
&\xrightarrow{\mathrm{C1} [3,2]} [[-3,-2],[3,4],[a, -1,2,-4,1, b]; [-1, 1, 1, -1]]
\\
&\xrightarrow{\text{(bps-2)}^{-1} [3,4]} [[a,-1,2,-2,1,b]; [-1, 1]]. 
\end{align*}

\noindent\textbf{Reduction of RHS of the CP move.}
\begin{align*}
  \text{RHS of the CP move}&
  =[[a, -1,-2,3,-4,2,4,5,-5,1,-3, b];[-1, -1, 1, -1, 1]]
\\
&\xrightarrow{\mathrm{F1} [4,5]}
[[-5,-4,-6],[4,2,6],[a, -1,-2,3,5,1,-3, b]; [-1, -1, 1, -1, 1, 1]]
\\
&\xrightarrow{\text{(ps-III)}^{-1} [6,4]} 
[[-4],[2],[a, -1,-2,3,4,1,-3,b]; [-1, -1, 1, 1]]
\\
&\xrightarrow{\mathrm{A1} [3,4]}
[[-3,-4],[2],[a, -1,-2,5,1,3,-5,4, b];[-1, -1, 1, -1, 1]]
\\
&\xrightarrow{\text{(bps-1)}^{-1} [4,3]} 
[[-3],[2],[a, -1,-2,3,1,b]; [-1, -1, 1]]
\\
&\xrightarrow{\mathrm{C1} [2,3]}
[[-2,-3],[2,4],[a, -1,3,-4,1, b]; [-1, 1, 1, -1]]
\\
&\xrightarrow{\text{(bps-2)}^{-1} [2,4]}
[[a, -1,2,-2,1, b]; [-1, 1]]
\end{align*}
Thus we have the assertion.


\begin{bibdiv}
\begin{biblist}

\bib{AK}{article}{
   author={Andersen, J\o rgen Ellegaard},
   author={Kashaev, Rinat},
   title={A TQFT from quantum Teichm\"uller theory},
   journal={Comm. Math. Phys.},
   volume={330},
   date={2014},
   number={3},
   pages={887--934},
   issn={0010-3616},
}

\bib{BP}{book}{
   author={Benedetti, Riccardo},
   author={Petronio, Carlo},
   title={Branched standard spines of $3$-manifolds},
   series={Lecture Notes in Mathematics},
   volume={1653},
   publisher={Springer-Verlag, Berlin},
   date={1997},
   pages={viii+132},
   isbn={3-540-62627-1},
}


\bib{BS0}{article}{
   author={Baseilhac, St\'ephane},
   author={Benedetti, Riccardo},
   title={Quantum hyperbolic invariants of $3$-manifolds with ${\rm
   PSL}(2,\Bbb C)$-characters},
   journal={Topology},
   volume={43},
   date={2004},
   number={6},
   pages={1373--1423},
   issn={0040-9383},
}

\bib{BS1}{article}{
   author={Baseilhac, St\'ephane},
   author={Benedetti, Riccardo},
   title={Classical and quantum dilogarithmic invariants of flat ${\rm
   PSL}(2,\Bbb C)$-bundles over $3$-manifolds},
   journal={Geom. Topol.},
   volume={9},
   date={2005},
   pages={493--569},
   issn={1465-3060},
}

\bib{BS2}{article}{
   author={Baseilhac, St\'ephane},
   author={Benedetti, Riccardo},
   title={Non ambiguous structures on $3$-manifolds and quantum symmetry
   defects},
   journal={Quantum Topol.},
   volume={8},
   date={2017},
   number={4},
   pages={749--846},
   issn={1663-487X},
}

\bib{regina}{misc}{
  author={Burton, Benjamin A.},
  author={Budney, Ryan},
  author={Pettersson, William},
  author={and others},
  title={Regina: Software for low-dimensional topology},
  note={\url{http://regina-normal.github.io/}},
  date={1999--2025}
}


\bib{C}{article}{
   author={Costantino, Francesco},
   title={A calculus for branched spines of $3$-manifolds},
   journal={Math. Z.},
   volume={251},
   date={2005},
   number={2},
   pages={427--442},
   issn={0025-5874},
}

\bib{snappy}{misc}{
  author={Culler, Marc},
  author={Dunfield, Nathan},
  author={Goerner, Matthias},
  author={Weeks, Jeffrey},
  title={SnapPy, a computer program for studying the geometry and topology of {$3$}-manifolds},
  note={Available at \url{http://snappy.computop.org}}
}

\bib{Ishii2}{article}{
   author={Ishii, Ippei},
   title={Moves for flow-spines and topological invariants of $3$-manifolds},
   journal={Tokyo J. Math.},
   volume={15},
   date={1992},
   number={2},
   pages={297--312},
   issn={0387-3870},
}



\bib{Mat}{book}{
   author={Matveev, Sergei},
   title={Algorithmic topology and classification of $3$-manifolds},
   series={Algorithms and Computation in Mathematics},
   volume={9},
   edition={2},
   publisher={Springer, Berlin},
   date={2007},
}

\bib{MST1}{article}{
   author={Mihalache, Serban Matei},
   author={Suzuki, Sakie},
   author={Terashima, Yuji},
   title={The Heisenberg double of involutory Hopf algebras and invariants of closed $3$-manifolds},
   journal={Algebr. Geom. Topol.},
   volume={24},
   date={2024},
   number={7},
   pages={3669--3691},
}

\bib{MST2}{article}{
   author={Mihalache, Serban Matei},
   author={Suzuki, Sakie},
   author={Terashima, Yuji},
   title={Quantum invariants of framed $3$-manifolds based on ideal triangulations},
   note={preprint (2022), arXiv:2209.07378},
}

\bib{MST}{article}{
   author={Muramatsu, Kohei},
   author={Suzuki, Sakie},
   author={Taguchi, Koki},
   title={On Matveev-Piergallini moves for branched spines. arXiv: 2405.18743.},
   note={preprint (2021), arXiv:math.GT/2405.18743},
}


\bib{Thurston}{book}{
  author={Thurston, William P.},
  title={The Geometry and Topology of Three-Manifolds},
  note={Princeton University lecture notes, 1978--1981},
  url={http://library.msri.org/books/gt3m/}
}
 
 \bib{TV}{article}{
   author={Turaev, V. G.},
   author={Viro, O. Ya.},
   title={State sum invariants of $3$-manifolds and quantum $6j$-symbols},
   journal={Topology},
   volume={31},
   date={1992},
   number={4},
   pages={865--902},
   issn={0040-9383},
}




\end{biblist}
\end{bibdiv}
\end{document}